\documentclass[11pt]{article}

\usepackage{amsfonts,amsmath,amsthm,amssymb,multicol}

\newtheorem{theorem}{Theorem}
\newtheorem{lemma}{Lemma}
\newtheorem{corollary}{Corollary}

\renewcommand{\emptyset}{\varnothing}

\newcommand\dis{\mathbf d}

\newcommand{\N}{\mathbb{N}}

\newcommand{\R}{\mathbb{R}}

\newcommand{\Par}{\overset{\mathbf{P}}{\longrightarrow}}

\title{\small\bf INVARIANCE OF FLUID LIMITS FOR THE\\
 SHORTEST REMAINING PROCESSING TIME AND\\
	SHORTEST JOB FIRST POLICIES}

\author{ {\small\sc 
H.\ Christian Gromoll\thanks{Research supported in part by NSF grant DMS
0707111} and Martin Keutel} \\
{\em\footnotesize University of Virginia} 
}

\begin{document}

\maketitle

\begin{abstract}
We consider a single-server queue with renewal arrivals and i.i.d.\ service
times, in which the server employs either the preemptive
Shortest Remaining Processing Time (SRPT) policy, or its non-preemptive
variant, Shortest Job First (SJF). We show that for given stochastic
primitives (initial condition, arrival and service processes), the model has
the same fluid limit under either policy. In particular, we conclude that the
well-known queue length optimality of preemptive SRPT is also achieved,
asymptotically on fluid scale, by the simpler-to-implement SJF policy. We also
conclude that on fluid scale, SJF and SRPT achieve the same performance with
respect to response times of the longest-waiting jobs in the system.
\end{abstract}



\section{Introduction}
\label{intro}

Schrage's \cite{Schrage1968} classic result asserts that for given stochastic
primitives (initial condition, arrival and service processes), the preemptive
Shortest Remaining Processing Time (SRPT) policy minimizes queue length at
all times, over all (work-conserving) policies. Recently, Gromoll, Kruk, and
Puha \cite{GrKruPuh2010} showed that, in the heavy-traffic regime, the
diffusion scaled queue length process under SRPT also achieves the optimal
lower bound $W(t)/x^*$ for all $t\ge0$, where $W(\cdot)$ is the diffusion
limit of the workload process and $x^*$ is the supremum of the support of the
service time distribution. 

Under SRPT, preemptive priority is given to the job that can be completed
first. More precisely, consider a single server queue with renewal arrivals
and i.i.d.\ service times, and let $\mathcal{I}(t)$ index, in arrival order,
those jobs that are in the queue at time $t$. For a job $i\in \mathcal{I}(t)$,
let $v_i(t)$ denote its {\em residual service time} at time $t$, or the
remaining amount of processing time required to complete this job. If
$j\in \mathcal{I}(t)$ is the smallest index such that $v_j(t)\le v_i(t)$ for
all $i\in \mathcal{I}(t)$, then under SRPT, $\frac{d}{dt}v_j(t+)=-1$ and
$\frac{d}{dt}v_i(t+)=0$ for all $i\in \mathcal{I}(t)\setminus j$. 

Due to its simple description and attractive optimality properties (see also
Smith \cite{Smith1956}), the SRPT policy and its relatives have received much
attention over the past several decades; see Schreiber \cite{Schreiber1993}
for a survey, as well as \cite{BansalHarchol2001, DoGrPu} and references
therein for more recent work. 

Despite its theoretical advantages, SRPT may not always be the best choice
in practice. Practical disadvantages include inexact job size information, the
need to continuously monitor the state of the queue and keep this state as a
continuum of priority classes, inefficiencies associated with frequent
switchovers, and long response times for large jobs. See for example
\cite{AM2005, NW2008}. 

These disadvantages have led researchers to consider a number of variants of
SRPT, all of which retain the essential feature of prioritizing smaller jobs
in some way.  Care is needed when reviewing the literature, as differing
terminology is in use. Usually, ``SRPT'' refers to the preemptive policy
defined above.  This policy may also be referred to as Shortest Time to
Completion First, or STCF.  

Shortest Job First (SJF) is the non-preemptive variant of SRPT. An SJF server
still selects the smallest job in the queue. But once begun, this job is
served to completion before the server makes another selection. This policy
may also be referred to as Shortest Job Next (SJN). There are also preemptive
versions of SJF and SJN, which differ from SRPT because they prioritize
based on initial rather than remaining job size. Nuyens and Wierman
\cite{NW2008} have introduced the class of $\epsilon$-SMART policies, which
include the preemptive policies mentioned above as well as policies that allow
for inexact job size information and finite sets of priority classes based on
job size. 

In this paper, we compare the preemptive SRPT policy to its non-preemptive SJF
variant.  Our analysis implies that, in many respects, these two policies
achieve the same performance (although not in all respects; see below).
Although SJF does not resolve all of the criticisms that have been raised for
SRPT, it does have important advantages from an implementation stand point. In
particular, there is no need to continuously monitor the state of the queue
and it avoids any penalties for interrupting service. Consequently,
practitioners interested in performance measures for which there is no
essential difference between the two policies may opt for the simpler SJF
policy, and this paper provides the theoretical justification for such a
choice.

\subsection{Main result}

Consider an $r$-indexed sequence of single server queues, each with renewal
arrivals, generally distributed i.i.d.\ service times, and random initial
state. In the $r$th model, let $\mathcal{Z}^r_p(\cdot)$ denote the
measure-valued state descriptor when the server employs preemptive SRPT, and
let $\mathcal{Z}^r(\cdot)$ denote the state descriptor under SJF. These
processes track the residual service times of all jobs, as they evolve
according to each respective policy; see Section \ref{sec:2} for details. Let
$\bar{\mathcal{Z}}^r_p(\cdot)$ and $\bar{\mathcal{Z}}^r(\cdot)$ denote these
processes under fluid scaling.  We will impose the standard asymptotic
assumptions \eqref{Rem,1}--\eqref{Rem,5} of Section \ref{sec:limits}. In
particular, there is a limiting random initial measure $\mathcal{Z}_0$, a
limiting arrival rate $\alpha$, and a limiting service time distribution
$\nu$, such that the limiting traffic intensity is at most one. 

In this setting, Down, Gromoll, and Puha \cite{DoGrPu} showed that as
$r\to\infty$, the sequence $\{\bar{\mathcal{Z}}_p^r(\cdot)\}$ converges in
distribution to a process $\mathcal{Z}^*(\cdot)$ that is almost surely an {\em
SRPT fluid model solution} for the data $(\alpha,\nu)$ and initial condition
$\mathcal{Z}^*(0)$, where $\mathcal{Z}^*(0)=\mathcal{Z}_0$ in distribution;
see Theorem 3.9 in \cite{DoGrPu}.

An SRPT fluid model solution for $(\alpha,\nu)$ with initial condition $\xi$
is a deterministic measure-valued function $\zeta(\cdot)$ satisfying
$\zeta(0)=\xi$ and a certain family of dynamic inequalities; see Section 2.2,
Definition 2.2 in \cite{DoGrPu}. Properties of SRPT fluid model solutions were
extensively studied in \cite{DoGrPu}. 

The main result of this paper is that the sequence of state descriptors under
the SJF policy has the same fluid limit:

\begin{theorem}
\label{upshot} 
Under the asymptotic assumptions \eqref{Rem,1}--\eqref{Rem,5}, the sequence
$\{\bar{\mathcal{Z}}^r(\cdot)\}$ converges in distribution to the measure
valued process $\mathcal{Z}^*(\cdot)$, where $\mathcal{Z}^*(\cdot)$ is almost
surely an SRPT fluid model solution for the data $(\alpha,\nu)$ and initial
condition $\mathcal{Z}^*(0)$, and $\mathcal{Z}^*(0)$ is equal in distribution
to $\mathcal{Z}_0$.
\end{theorem}

\subsection{Implications for comparing performance}

Since the queue length processes under SJF and SRPT are given by the total
mass of the state descriptors $\mathcal{Z}^r(\cdot)$ and $\mathcal{Z}^r_p(\cdot)$
respectively, an immediate consequence of the above theorem is that,
asymptotically on fluid scale, SJF enjoys the same queue length optimality as
SRPT. 

A second consequence relates to state-dependent response times. A job's
state-dependent response time is the time until it exits the system,
conditional on the state of the system (the configuration of all current
residual service times) when it arrives. The fluid limit
$\mathcal{Z}^*(\cdot)$ can be used to calculate a fluid approximation $s(x)$
to the state-dependent response time of a job of size $x$, as a function of an
initial state $\xi$; see \cite{DoGrPu2, DoGrPu}.  Thus Theorem \ref{upshot}
implies that, asymptotically on fluid scale, a job of size $x$ arriving to a
queue in state $\xi$ will have the same state-dependent response time under
SJF or SRPT. Moreover, the system will be in the same state $\xi$ under either
SJF or SRPT. 

Note that SJF and SRPT exhibit differing performance for the {\em mean}
response time of a job of size $x$ (that is, averaged over all possible states
encountered on arrival when the system is in steady state).  Bansal and
Harchol-Balter \cite{BansalHarchol2001} show for heavy-tailed service time
distributions, that the mean response time under SRPT is orders of magnitude
lower than under SJF, while the very largest jobs have smaller mean response
times under SJF. See also Harchol-Balter, Sigman, and Wierman \cite{HSW2002}
for explicit formulae for mean response times  under SRPT, SJF, and several
other protocols. 

The reason SJF and SRPT have the same fluid approximation for state-dependent
response times, is that fluid scaling compresses time such that individual
jobs -- even large ones -- exit almost instantaneously once they begin
service. In fact, any arriving job with service time smaller than the largest
job to begin service during the current busy period, or {\em frontier}, will
have negligible response time on fluid scale. Jobs with service times larger
than the frontier wait a positive time on fluid scale to begin service for the
first time. 

Thus, for either policy, the fluid limit shows that response times for jobs
larger than the frontier are orders of magnitude longer than for those below
it. In particular, the fluid limit captures the profile of reponse times only
for those jobs that experience the very longest response times in the system,
namely, those that are bigger than the frontier when they arrive.
And given the initial state, these response times are the same under SJF or
SRPT. This conclusion is consistent with Nuyens and Zwart \cite{NZ2006}, where
it is shown that the steady-state response time has the same asymptotic decay
rate under SJF and SRPT.

\subsection{Notation}
\label{sec:1}
For real numbers $a$ and $b$, let $a \wedge b = \min \{a,b\}$ and $a \vee b =
\max \{a,b\}$.  Denote the natural numbers $1, 2, \dots$ by $\mathbb{N}$ and
the non-negative real numbers $[0,\infty)$ by $\R_+$.  The measure that puts
one unit of mass at $x \in \R_+$ is written $\delta_x$, and the meaure
$\delta_x^+$ is $\delta_x$ if $x > 0$ and the zero measure (denoted ${\bf 0}$)
if $x = 0$. 

Let ${\bf M}$ denote the Polish space of finite, nonnegative Borel measures on
$\R_+$, endowed with the weak topology \cite{Pro1956}.   A metric inducing
this topology is the following Prohorov metric on ${\bf M}$. For $A \subset
\R_+$, let $A^{\epsilon} = \{ y \in \R_+ : \inf_{x\in A} |x - y| < \epsilon
\}$.  Then for $\xi$, $\zeta \in {\bf M}$, define
\begin{multline} 
  {\bf d}[\xi, \zeta] =
  \inf\{\epsilon > 0: \xi(B) \le \zeta(B^\epsilon) + \epsilon
  \text{ and } \zeta(B) \le \xi(B^\epsilon) + \epsilon \\
  \text{  for all closed } B \subset \mathbb{R_+} \}.  \notag 
\end{multline}

For $\zeta \in {\bf M}$ and a $\zeta$-integrable function $g:\R_+ \to \R$,
define $\langle g, \zeta \rangle = \int_{\R_+} g(x)\zeta(dx)$.  Let $\chi(x) =
x$ be the identity function, and for any set $A$, let $1_A$ denote its
indicator function. For a real-valued function $f$, write $f^+$ and $f^-$ for
its positive and negative parts respectively. Unless otherwise specified, all
processes are assumed to be right continuous with finite left limits (rcll).
In particular, all measure-valued processes in this paper take values in the
Skorohod space $\mathbf{D}([0,\infty),\mathbf{M})$ endowed with the
$J_1$-topology.  Convergence in probability is denoted $\Par$,  weak
convergence of elements of $\mathbf{M}$ is denoted
$\overset{w}{\longrightarrow}$, and convergence in distribution of random
objects is denoted $\Rightarrow$.  We adopt the convention that a sum of the
form $\sum_{i = n}^m$ with $n>m$, or a sum over an empty set of indices equals
zero. Finally, we also define $\min\emptyset = \inf\emptyset = \infty.$

\section{The model: SJF and SRPT}
\label{sec:2}
This section lays out the precise definitions of our model, including the two
policies under consideration.  We will compare the SJF and SRPT policies
pathwise, using the same sample space $(\Omega,\mathcal{F},\mathbf{P})$ and
same stochastic primitives.  The stochastic primitives consist of an exogenous
arrival process, a sequence of service times, and an initial condition.

{\bf Stochastic primitives.} The arrival process $E(\cdot)$ is a rate $\alpha
\in (0, \infty)$ (possibly delayed) renewal process. For $t\ge0$, $E(t)$ is
the number of jobs that have arrived to the queue during $(0, t]$.  The
initial service times of these jobs are taken from an i.i.d.\ sequence
$\{w_j:j=1,2,\dots\}$ of strictly positive random variables, with
distribution $\nu$ having finite mean. Define the traffic intensity $\rho =
\alpha \langle \chi, \nu \rangle$.   

The initial condition is given by a non-negative integer-valued random
variable $Z_0$ with finite mean, and a sequence of strictly positive random
variables $\{w_j:j=\dots,-2,-1\}$. The random variable $Z_0$ represents
the number of jobs in the queue at time $0$, and their initial service times
are $\{w_{-Z_0},\dots,w_{-1}\}$. 

Define the random index set $J=\{j\in\mathbb{Z}:j\ge-Z_0\}\setminus\{0\}$ and,
for each $t\ge0$, define the random index set $J(t)=J\cap\{-Z_0,\dots,E(t)\}$.
Then $J$ indexes all jobs that will ever be in the system, and $J(t)$ indexes
all jobs that have been in the system by time $t$. Note that neither set ever
contains the index $0$, and that $J(t)=\varnothing$ if $Z_0=E(t)=0$. Note that
the present indexing differs slightly from that used in \cite{DoGrPu}, but
this does not affect our analysis. 

{\bf Workload.}  Define a measure-valued load process $\mathcal{V}(\cdot)$ and
a load process $V(\cdot)$ by 
\begin{align*} 
     \mathcal{V}(t)   &= \sum_{j = 1}^{E(t)} \delta_{w_j}, \\
     V(t)             &= \left \langle \chi, \mathcal{V}(t)\right\rangle,
\end{align*}	
for all $t\ge0$. Define the workload process $W(\cdot)$ and cumulative idle
time process $I(\cdot)$ by 
\begin{equation} \begin{aligned} 
     W(0)             &=  \sum_{j = -Z_0}^{-1} w_j, \\
     I(t)             &= \sup_{s \leq t} [W(0) + V(s) - s]^-,  \\
     W(t)             &= W(0) + V(t) - t + I(t).  \end{aligned} 
\label{Def1}   
\end{equation}

Next, we define the SJF and SRPT policies, as well as the corresponding state
descriptors that describe the evolution of the system under each policy.

{\bf Shortest Job First.} In order to define the residual service times under
SJF, we first define the succesive times $\{\gamma_n\}$ at which the server
begins processing a new job. Let $\gamma_0=y_0=j_0=0$ and define the set
$J_0=\varnothing$.  Now define inductively, for $n=1,2,\dots,$ 
\begin{equation}\label{gamma.y}
\begin{aligned}
  \gamma_n &=\inf\{t\ge \gamma_{n-1}+y_{n-1}:Z_0+E(t)\ge n\}, \\
  y_n	 &= \min\{w_j:j\in J(\gamma_n)\setminus J_{n-1}\}, \\
  j_n &  =\min\{j\in J(\gamma_n)\setminus J_{n-1}:w_j=y_n\}, \\
  J_n & = J_{n-1}\cup\{j_n\}.
\end{aligned}
\end{equation}
For each $n=1,2,\dots$, the time $\gamma_n$ is the $n$th time at which the
server begins processing a new job, $y_n$ is the service requirement of this
job, $j_n$ is its index, and $J_n$ is the set indexing all jobs that have
begun service by time $\gamma_n$. For each $j\in J$, let
\[
n(j)=\min\{n\in\mathbb{N}:j\in J_n\}.
\]
Then $n(j)$ is the overall order in which job $j$ is served, and the start
time of job $j$ may be defined as $\gamma_{n(j)}$ (with $\gamma_\infty$
interpreted as $\infty$, which may occur for example in a supercritical model
in which job $j$ never begins service).  In particular, note that
$y_{n(j)}=w_j$ for all $j\in J$ such that $\gamma_{n(j)}<\infty$.  At time
$t\ge0$, define the attained service of job $j\in J$ as
\[
s_j(t)=(t-\gamma_{n(j)})^+\wedge w_j,
\]
the cumulative service provided as
\[ S(t)=\sum_{j\in J(t)}s_j(t),
\]
and define the residual service time of job $j\in J$ as 
\[
w_j(t)=w_j-s_j(t). 
\]

We now define the state descriptor under SJF as the measure-valued process
\[ \mathcal{Z}(t)  = \sum_{j \in J(t)} \delta_{w_j(t)}^+, \quad
t\ge0. 
\]
Finally, for each $t\ge0$, let $F(t)$ denote the largest job started by the SJF
server by time $t$.  That is,  
\begin{equation*} F(t) = \max_{n \in \N}
  \left\{y_n: \gamma_n \leq t \right\}.
\end{equation*}

{\bf Shortest Remaining Processing Time.} For job $j\in J$, let $v_j(t)$
denote its residual service time at time $t$ when the server employs the SRPT
policy; we refer the reader to \cite{DoGrPu} for a detailed definition. Then
the measure-valued state descriptor under SRPT is given by
\[ \mathcal{Z}_p(t) = \sum_{j \in J(t)} \delta_{v_j(t)}^+, \quad t\ge0.
\]
Throughout the paper, the subscript $p$ for ``preemptive'' is used to
distinguish performance processes under SRPT from their analogues under SJF.
In \cite{DoGrPu}, the \emph{left edge process} of the state descriptor
$\mathcal{Z}_p(\cdot)$ is defined as 
\begin{equation} L_p(t) = \sup\{x \in \R_+: \langle 1_{[0, x)},
  \mathcal{Z}_p(t)\rangle = 0\},\quad t \ge0. \notag
\end{equation}
The current residual service time process is then defined, for all $t \ge 0$, by
\begin{equation} C_p(t) = 
                 \begin{cases}
                  L_p(t),           &\text{if $\mathcal{Z}_p(t) \neq {\bf 0}$}\\
                  0,                &\text{otherwise,}
                 \end{cases} \notag \end{equation} 
and the SRPT frontier process is defined by
\begin{equation} F_p(t) = \sup_{0\leq s\leq t} C_p(s), \quad t\ge0.
 \notag 
\end{equation}  

\section{Pathwise comparison}
\label{sec:3}
In this section, we establish several results that relate SJF to SRPT on a
pathwise basis. The main goal is to show that SJF and SRPT have the same
frontier process almost surely. This will provide the key element in the proof
of Theorem \ref{upshot} in Section \ref{sec:limits}.

First, we show that the sum of the residual service times under SJF, as
defined above, equals the workload defined in \eqref{Def1}.

\begin{lemma}\label{workload}
  Almost surely, for all $t\ge0$, $S(t)=t-I(t)$, and
  $W(t)=\sum_{j\in J(t)}w_j(t)$.
\end{lemma}

\noindent
\begin{proof} Clearly $S(t)\le W(0)+V(t)$ for all $t\ge0$. Recall that a
regular point of a function is a point at which it is differentiable.  We
first show that, for all regular points $t$, 
\begin{equation}
  \dot S(t)=0 \quad\text{if and only if}\quad S(t)=W(0)+V(t).  
  \label{flatwhen}
\end{equation}
Note that almost surely for all $j\in J$, the derivative $\dot
s_j(t)\in\{0,1\}$ for all regular points $t$ of $s_j(\cdot)$, and $\dot
s_j(t)=1$ if and only if $t\in(\gamma_{n(j)},\gamma_{n(j)}+y_{n(j)})$.
Further, the definitions above imply that the finite elements of $\{n(j):j\in
J\}$ are distinct and the time intervals
$\{(\gamma_n,\gamma_n+y_n):n\in\mathbb{N}\}$ are all disjoint. This implies
that $\dot S(t)\in\{0,1\}$ for all regular points $t$. (Note that if $t$ is a
regular point of $S(\cdot)$ but not of all $s_j(\cdot)$, $j\in J(t)$, then we
must have $\dot S(t)=1$ anyway). Moreover, $\dot S(t)=0$ if and only if
$t<\gamma_1$ or $t\in (\gamma_n+y_n,\gamma_{n+1})$ for some $n\in\mathbb{N}$.
If $t<\gamma_1$, then $S(t)=0=W(0)+V(t)$, and the converse is also true. If
$t\in(\gamma_n+y_n,\gamma_{n+1})$ for some $n\in\mathbb{N}$, then
\eqref{gamma.y} implies $Z_0+E(t)=n$, which implies that $s_j(t)=w_j$ for all
$j\in J(t)$. Thus, $S(t)=W(0)+V(t)$, and the converse is also true. This
proves \eqref{flatwhen}. 

Observe that almost surely, $S(\cdot)$ is continuous, starts at zero, has
derivative in $\{0,1\}$ at regular points, is bounded above by $W(0)+V(\cdot)$,
and satisfies \eqref{flatwhen} at all regular points $t$. Since
$W(0)+V(\cdot)$ is right-continuous and piecewise constant, this uniquely
determines $S(\cdot)$ almost surely.
 
As is well known, $\dot I(t)\in\{0,1\}$ at all regular points $t$, and
$\dot I(t)=1$ if and only if $I(t)=(W(0)+V(t)-t)^-$ and $I(t)>0$. So
$\frac{d}{dt}(t-I(t))=0$ if and only if $t-I(t)=W(0)+V(t)$. Also, $t-I(t)\le
W(0)+V(t)$ for all $t\ge0$. Since $\cdot-I(\cdot)$ is continuous and starts
from zero, it is also uniquely determined by $W(0)+V(\cdot)$ almost surely. We
conclude that $S(t)=t-I(t)$ for all $t\ge0$ almost surely.

The second part follows directly from the first part and equation \eqref{Def1}, since
\[ 
W(t)=W(0)+V(t)-t+I(t)=\sum_{j\in J(t)}w_j
	-S(t)=\sum_{j\in J(t)}w_j(t).
\]
\end{proof}
Note that by Lemma \ref{workload}, for all $t\ge0$,
\begin{equation}
  W(t)=\langle\chi,\mathcal{Z}(t)\rangle.
  \label{workload-chi}
\end{equation}

Next, we prove a result describing the workload at or above the frontier under
SJF, at certain random times.  To that end, define the random set 
\[
	\mathcal{J}=\{j\in J: \gamma_{n(j)}<\infty\text{ and } 
		w_j=F(\gamma_{n(j)})\}.
\]
These are frontier jobs, with service time equal to the frontier when they are
started.  The set of start times of frontier jobs is
$\mathcal{T}=\{\gamma_{n(j)}:j\in\mathcal{J}\}$. Let
$\{\tau_k:k\in\mathbb{N}\cap[0,|\mathcal{J}|]\}$ be the ordering of
$\mathcal{T}$ such that $\tau_k<\tau_{k+1}$ for all $k<|\mathcal{J}|$.  Note
that $\mathcal{J}$ is an infinite set in most settings of interest. But it can
happen that $|\mathcal{J}|<\infty$, for example when $\rho>1$ or when there
are initial jobs larger than the supremum of the support of $\nu$.

\begin{lemma}\label{interStartTimes-n} (i) Almost surely, for all $k \in
  \mathbb{N}$ with $k \le|\mathcal{J}|$, 
\begin{equation}
\langle \chi 1_{[F(\tau_k),\infty)},\mathcal{Z}(\tau_k)\rangle=W(\tau_k).
  \label{WisAbove}
\end{equation}
(ii) Almost surely, for all $k < |\mathcal{J}|$ and $t\in(\tau_k,\tau_{k+1}],$
  \begin{equation}
  \langle\chi1_{[F(\tau_k),\infty)},\mathcal{Z}(t)\rangle
  =W(\tau_k) -F(\tau_k) +  \sum_{j=E(\tau_k)+1}^{E(t)}w_j1_{\{w_j\ge
			F(\tau_k)\}}.
	\label{WaboveEqn}
	\end{equation}
	Moreover, if $k=|\mathcal{J}|<\infty$ and $t\in(\tau_k,\infty)$, then
	\eqref{WaboveEqn} holds as well.
\end{lemma}

\begin{proof} (i) Fix $k\in \mathbb{N}\cap[0, |\mathcal{J}|]$.  By definition, $\tau_k=\gamma_{n(j_0)}$
for exactly one $j_0\in\mathcal{J}$, and
$F(\tau_k)=F(\gamma_{n(j_0)})=w_{j_0}$.  Consider any $j\in J(\tau_k)$. If
$\gamma_{n(j)}<\tau_k$, then \eqref{gamma.y} implies that
$\gamma_{n(j)}+y_{n(j)}\le\gamma_{n(j_0)}$, and so $w_j(\tau_k)=0$. On the
other hand, if $\gamma_{n(j)}\ge\tau_k$, then $w_j(\tau_k)=w_j$ and
\eqref{gamma.y} implies that $j\in J(\gamma_{n(j_0)})\setminus J_{n(j_0)-1}$.
In particular, $w_j\ge y_{n(j_0)}=w_{j_0}$. So at time $\tau_k$, all arrived
jobs have residual service time either equal to zero or bounded below by
$F(\tau_k)$. This implies that the first right-hand term in
\begin{equation*}
  W(\tau_k)=\langle\chi 1_{[0,F(\tau_k))},\mathcal{Z}(\tau_k)\rangle
  +\langle \chi 1_{[F(\tau_k),\infty)},\mathcal{Z}(\tau_k)\rangle
\end{equation*}
equals zero and proves \eqref{WisAbove}.
\newline \noindent (ii) To prove \eqref{WaboveEqn}, fix $k < |\mathcal{J}|$ and $t\in(\tau_k,\tau_{k+1}]$.  Then
\begin{equation}\label{interStartTimes-n1} 
  \langle\chi1_{[F(\tau_k),\infty)},\mathcal{Z}(t)\rangle
  = \sum_{j\in J(t)}w_j(t)1_{\{w_j(t)\ge F(\tau_k)\}}.
\end{equation}
Observe that $w_{j_0}(t)<w_{j_0}=F(\tau_k)$,
because $t>\gamma_{n(j_0)}$ and $\gamma_{n(j_0)}\in\mathcal{T}$. So
\begin{equation}
  w_{j_0}(t)1_{\{w_{j_0}(t)\ge F(\tau_k)\}}=0,
  \label{jnought1}
\end{equation}
while
\begin{equation}
  \label{jnought2}
  w_{j_0}(\tau_k)1_{\{w_{j_0}(\tau_k)\ge F(\tau_k)\}}=F(\tau_k).
\end{equation}

Now consider $j\in J(t)\setminus j_0$. We will show that 
\begin{equation}\label{jnoughtToShow-a}
	w_j(t)1_{\{w_j(t)\ge F(\tau_k)\}} = w_j(\tau_k)1_{\{w_j(\tau_k)\ge F(\tau_k)\}}. 
\end{equation}

Case one: If $\gamma_{n(j)} < \tau_k$, then $w_{j}(\tau_k) = 0$ by
\eqref{gamma.y}.  Since $w_{j}(\cdot)$ is nonincreasing and $t > \tau_k,
w_j(t) = 0$ as well and so \eqref{jnoughtToShow-a} holds.  

Case two: If $\gamma_{n(j)}\in (\tau_k,\tau_{k+1})$, then $j \notin
\mathcal{J}$ and so $w_j<F(\gamma_{n(j)})$. Since $F(\cdot)$ is
right-continuous, piecewise constant, and can only jump at (a subset of) times
$\tau\in\mathcal{T}$, this implies $w_j < F(\tau_k).$  Then $1_{\{w_j(t) \ge
F(\tau_k)\}} =  1_{\{w_j(\tau_k) \ge F(\tau_k)\}} =0$ and equation
\eqref{jnoughtToShow-a} holds.

Case three: If $\gamma_{n(j)} \ge \tau_{k+1}$, then $\gamma_{n(j)} \ge t$ and
so $s_j(t) = 0$, that is, $w_j(t) = w_j(\tau_k)$ and equation
\eqref{jnoughtToShow-a} holds.  

This proves \eqref{jnoughtToShow-a}.  Combining \eqref{interStartTimes-n1}
with \eqref{jnought1} and \eqref{jnoughtToShow-a} and noting that if $j >
E(\tau_k)$ then $w_j(\tau_k) = w_j,$ we obtain
\begin{equation*}
  \langle\chi1_{[F(\tau_k),\infty)},\mathcal{Z}(t)\rangle
  = \sum_{j\in J(\tau_k)\setminus j_0}w_j(\tau_k)1_{\{w_j(\tau_k)\ge F(\tau_k)\}}
  +\sum_{j=E(\tau_k)+1}^{E(t)}w_j1_{\{w_j\ge F(\tau_k)\}}.
\end{equation*}
By \eqref{jnought2}, including the $j_0$-term in the first sum is compensated by
subtracting $F(\tau_k)$, which yields 
\begin{equation*}
  \langle\chi1_{[F(\tau_k),\infty)},\mathcal{Z}(t)\rangle
  = \langle\chi 1_{[F(\tau_k),\infty)},\mathcal{Z}(\tau_k)\rangle -F(\tau_k)
  +\sum_{j=E(\tau_k)+1}^{E(t)}w_j1_{\{w_j\ge F(\tau_k)\}}.
\end{equation*}
Applying \eqref{WisAbove} proves \eqref{WaboveEqn}.

It remains to show \eqref{WaboveEqn} when $k=|\mathcal{J}|<\infty$ and
$t\in(\tau_k,\infty)$. The argument is identical except for Cases two and
three in the proof of \eqref{jnoughtToShow-a}. Case two becomes
$\gamma_{n(j)}\in (\tau_k,\infty)$. Then $j\notin\mathcal{J}$ and
$w_j<F(\gamma_{n(j)})$. Since $F(\cdot)$ is constant on $[\tau_k,\infty)$,
$w_j<F(\tau_k)$ and \eqref{jnoughtToShow-a} follows as before. Case three
becomes $\gamma_{n(j)}=\infty$. Then $w_j(\tau_k)=w_j(t)=w_j$, which implies
\eqref{jnoughtToShow-a}. The remainder of the proof follows from
\eqref{interStartTimes-n1}--\eqref{jnoughtToShow-a} as before.  \end{proof}

We need a similar result for SRPT, proved in the next two lemmas. For each
$j\in J$, the SRPT start time of job $j$ is defined $\eta_j=\inf\{t\ge0:\dot
v_j(t)=-1\}$.  Let 
\[
\mathcal{J}_p=\{j\in J:\eta_j<\infty \text{ and } w_j=F_p(\eta_j)\},
\]
let $\mathcal{T}_p=\{\eta_j:j\in\mathcal{J}_p\}$, and let
$\{\sigma_k:k\in\mathbb{N}\cap[0,|\mathcal{J}_p|]\}$ be the ordering of
$\mathcal{T}_p$ such that $\sigma_k<\sigma_{k+1}$ for all $k<|\mathcal{J}_p|$.

\begin{lemma}\label{CF}
  Almost surely, (i) if $t\in\mathcal{T}_p,$ then $C_p(t)=F_p(t)$, and
(ii) if $C_p(t)=F_p(t) >0$, then $t\in\mathcal{T}_p$.
\end{lemma}

\begin{proof} 
Fix $t \in \mathcal{T}_p$, so $t = \eta_j$ for some $j \in
\mathcal{J}_p.$  Then $w_j = F_p(\eta_j).$ Evidently $v_j(t) =
w_j > 0$, so $\mathcal{Z}_p(t) \neq {\bf 0}$, and therefore $C_p(t) =
L_p(t).$ Now suppose $\langle 1_{[0, w_j)}, \mathcal{Z}_p(t)\rangle > 0.$
Then there is at least one job in the queue at time $t$ with residual service
time $y\in(0, w_j)$, contradicting the assumption that SRPT starts job
$j$ at $t$.  This shows that $C_p(t) \geq F_p(t)$.  Since $F_p(t) \ge
C_p(t)$ by definition, this establishes (i).

For (ii), fix $t$ such that $C_p(t)=F_p(t) >0$. By \cite{DoGrPu} Lemma 5.3(i),
$\langle 1_{[0, C_p(t))}, \mathcal{Z}_p(t)\rangle = 0$ almost surely.
Additionally, noting differences in the indexing of jobs, \cite{DoGrPu} Lemma
5.3(ii) implies that for all $j \in J$ such that $w_j>F_p(t)$, $v_j(s) = w_j$
for all $s \in [0, t]$. Since  $C_p(t) > 0$, $\mathcal{Z}_p(t) \neq {\bf 0}$
and so $\langle 1_{[C_p(t), \infty)}, \mathcal{Z}_p(t)\rangle >0$.  Thus the
set $\{j \in J(t): C_p(t)= v_j(t) = w_j\}$ is nonempty.  Let $j^*$ be the
smallest element of this set. Almost surely, there exists
$\epsilon\in(0,C_p(t))$ such that $E(t + \epsilon) = E(t)$, that is, no new
jobs arrive during $(t, t + \epsilon]$.  Then $\frac{d}{ds}v_{j^*}(s)=0$ on
$[0,t)$ and $\frac{d}{ds}v_{j^*}(s) = -1$ on $(t, t + \epsilon)$, so
$\eta_{j^*}=\inf \{s \geq 0 : \dot{v}_{j^*}(s) = -1\} = t$.  Thus $w_{j^*} =
F_p(t) = F_p(\eta_{j^*})$ and hence $j^* \in \mathcal{J}_p$ and $t
\in \mathcal{T}_p$.  \end{proof}

\begin{lemma}\label{interStartTimes-p}
  (i) Almost surely, for all $k \in \mathbb{N}$ with $k \le|\mathcal{J}_p|$, 
  \begin{equation}\label{pWisAbove}
  \langle\chi1_{[F_p(\sigma_k),\infty)},\mathcal{Z}_p(\sigma_k)\rangle
  =W(\sigma_k).
	\end{equation}
  (ii) Almost surely, for all $k<|\mathcal{J}_p|$ and $t\in(\sigma_k,\sigma_{k+1}]$,
\begin{equation}\label{pWaboveEqn}
  \langle\chi1_{[F_p(\sigma_k),\infty)},\mathcal{Z}_p(t)\rangle
  =W(\sigma_k) -F_p(\sigma_k) +  \sum_{j=E(\sigma_k)+1}^{E(t)}w_j1_{\{w_j\ge
			F_p(\sigma_k)\}}.
	\end{equation}
	Moreover, if $k=|\mathcal{J}_p|<\infty$ and $t\in(\sigma_k,\infty)$, then
	\eqref{pWaboveEqn} holds as well.
\end{lemma}

\begin{proof} 
(i) Fix $k\in \mathbb{N}\cap[0, |\mathcal{J}_p|]$.  By \cite{DoGrPu} Lemma 5.3 (i), $\langle 1_{[0, C_p(t))},
\mathcal{Z}_p(t)\rangle = 0$ for all $t$ almost surely.  Since $\sigma_k \in
\mathcal{T}_p$, $F_p(\sigma_k) = C_p(\sigma_k)$ by Lemma $\ref{CF}$, and it
follows that  $\langle 1_{[0, F_p(\sigma_k))},\mathcal{Z}_p(\sigma_k)\rangle =
0.$ This implies that the first right-hand term in
\begin{equation*}
  W(\sigma_k)=\langle\chi 1_{[0,F_p(\sigma_k))},\mathcal{Z}_p(\sigma_k)\rangle
  +\langle \chi 1_{[F_p(\sigma_k),\infty)},\mathcal{Z}_p(\sigma_k)\rangle
\end{equation*}
equals zero and proves \eqref{pWisAbove}.
\newline\noindent (ii) To prove \eqref{pWaboveEqn}, fix $k < |\mathcal{J}_p|$ and fix  $t\in(\sigma_k,\sigma_{k+1}]$.  Then
\begin{equation}\label{interStartTimes-p1} 
  \langle\chi1_{[F_p(\sigma_k),\infty)},\mathcal{Z}_p(t)\rangle
  = \sum_{j\in J(t)}v_j(t)1_{\{v_j(t)\ge F_p(\sigma_k)\}}.
\end{equation}
Let $j_0\in\mathcal{J}_p$ be the job started at $\sigma_k$. Then
$w_{j_0}=F_p(\sigma_k)$ and $v_{j_0}(t)<w_{j_0}$ since $t>
\sigma_k=\eta_{j_0}$. Thus,
\begin{equation}
  v_{j_0}(t)1_{\{v_{j_0}(t)\ge F_p(\sigma_k)\}}=0,
  \label{jnought1a}
\end{equation}
while
\begin{equation}
  \label{jnought2a}
  v_{j_0}(\sigma_k)1_{\{v_{j_0}(\sigma_k)\ge F_p(\sigma_k)\}}=F_p(\sigma_k).
\end{equation}

  Now consider $j\in J(t)\setminus j_0$. We will show that 
\begin{equation}\label{jnoughtToShow}
	v_j(t)1_{\{v_j(t)\ge F_p(\sigma_k)\}} =
      v_j(\sigma_k)1_{\{v_j(\sigma_k)\ge F_p(\sigma_k)\}}.
\end{equation}

Case one: If $\eta_j < \sigma_k$, then $\dot{v}_j(t) = -1$ on $(\eta_j, \eta_j
+ \epsilon)$ for some $\epsilon > 0$, so $v_j(\sigma_k) < w_j \le F_p(\eta_j)
\le F_p(\sigma_k).$  Since $v_j(\cdot)$ is nonincreasing and $t > \sigma_k$,
we have $1_{\{v_j(t)\ge F_p(\sigma_k)\}} = 1_{\{v_j(\sigma_k)\ge
F_p(\sigma_k)\}} = 0,$ and \eqref{jnoughtToShow} holds.  

Case two: If $\eta_j\in(\sigma_k,\sigma_{k+1})$, then $j \notin \mathcal{J}_p$
and so $w_j<F_p(\eta_j)$.  Since $F_p(\cdot)$ is constant on
$[\sigma_k,\sigma_{k+1})$, this implies $v_j(t) \le v_j(\sigma_k)\le w_j <
F_p(\sigma_k)$. Then $1_{\{v_j(t) \ge F_p(\sigma_k)\}} = 1_{\{v_j(\sigma_k) \ge F_p(\sigma_k)\}} =0$ and \eqref{jnoughtToShow} holds.

Case three: If $\eta_j \ge \sigma_{k+1}$, then $\eta_j\ge t$ and so $v_j(t) =
v_j(\sigma_k)$ and \eqref{jnoughtToShow} holds.  

This proves \eqref{jnoughtToShow}.  Combining \eqref{interStartTimes-p1} with
\eqref{jnought1a} and \eqref{jnoughtToShow} and noting that if $j >
E(\sigma_k)$ then $v_j(\sigma_k) = w_j,$ we obtain 
\begin{equation*}
  \langle\chi1_{[F_p(\sigma_k),\infty)},\mathcal{Z}_p(t)\rangle
  = \sum_{j\in J(\sigma_k)\setminus j_0}v_j(\sigma_k)1_{\{v_j(\sigma_k)\ge F_p(\sigma_k)\}}
  +\sum_{j=E(\sigma_k)+1}^{E(t)}w_j1_{\{w_j\ge F_p(\sigma_k)\}}.
\end{equation*}
By \eqref{jnought2a}, including the $j_0$-term in the first sum is compensated by
subtracting $F_p(\sigma_k)$, which yields 
\begin{equation*}
  \langle\chi1_{[F_p(\sigma_k),\infty)},\mathcal{Z}_p(t)\rangle
  = \langle\chi 1_{[F_p(\sigma_k),\infty)},\mathcal{Z}_p(\sigma_k)\rangle -F_p(\sigma_k)
  +\sum_{j=E(\sigma_k)+1}^{E(t)}w_j1_{\{w_j\ge F_p(\sigma_k)\}}.
\end{equation*}
Applying \eqref{pWisAbove} proves \eqref{pWaboveEqn}.

It remains to show \eqref{pWaboveEqn} when $k=|\mathcal{J}_p|<\infty$ and
$t\in(\sigma_k,\infty)$. The argument is identical except for Cases two and
three in the proof of \eqref{jnoughtToShow}. Case two becomes
$\eta_j\in (\sigma_k,\infty)$. Then $j\notin\mathcal{J}_p$
and $w_j<F_p(\eta_j)$. Since $F_p(\cdot)$ is constant on $[\sigma_k,\infty)$,
$w_j<F_p(\sigma_k)$ and \eqref{jnoughtToShow} follows as before. Case three
becomes $\eta_j=\infty$. Then $v_j(\sigma_k)=v_j(t)=w_j$, which implies
\eqref{jnoughtToShow}. The remainder of the proof follows from
\eqref{interStartTimes-p1}--\eqref{jnoughtToShow} as before.  \end{proof}

We now show that SJF and SRPT have the same start times of frontier jobs, and
coinciding frontiers at those times. This is the principal technical result of
the paper.

\begin{lemma}\label{startTimes}
  Almost surely, $\mathcal{T}=\mathcal{T}_p$ and $F(\tau)=F_p(\tau)$ for all
  $\tau\in\mathcal{T}$.
\end{lemma}

\begin{proof} 
Clearly, $\sigma_1=\inf\{t\ge0:Z_0+E(t)\ge 1\}=\gamma_1$.  For the SJF
queue, $w_{j_1}=y_1=F(\gamma_1)$, so $j_1\in\mathcal{J}$ and thus
$\gamma_1\in\mathcal{T}$, implying $\tau_1=\gamma_1$. Conclude that
$\tau_1=\sigma_1$. Moreover, $F(\tau_1)=\min\{w_j:j\in
J(\tau_1)\}=F_p(\tau_1)$. Now fix $k<|\mathcal{J}|\wedge|\mathcal{J}_p|$ and
assume that $\tau_l=\sigma_l$ and $F(\tau_l)=F_p(\tau_l)$ for all
$l=1,\dots,k$.  We will show that $\tau_{k+1}=\sigma_{k+1}$ and
$F(\tau_{k+1})=F_p(\tau_{k+1})$.  

Note first that by definition of the start times, $\mathcal{Z}(\tau_{k+1})\neq
{\bf 0}$ and \newline $\mathcal{Z}_p(\sigma_{k+1})\neq {\bf 0}$. Since both
queues have the same workload process $W(\cdot)$, this implies that
$W(\tau_{k+1})\wedge W(\sigma_{k+1})>0$, and so $\mathcal{Z}(\sigma_{k+1})\neq
{\bf 0}$ and $\mathcal{Z}_p(\tau_{k+1})\neq {\bf 0}$.  We first show that
$\tau_{k+1}=\sigma_{k+1}$. 

Suppose that $\sigma_{k+1}>\tau_{k+1}$.  Then
$\tau_{k+1}\in(\sigma_k,\sigma_{k+1})$, so by Lemma \ref{interStartTimes-p}
and the induction hypothesis,
\begin{equation*}
  \langle \chi 1_{[F_p(\sigma_k),\infty)},\mathcal{Z}_p(\tau_{k+1})\rangle
  = W(\tau_k)-F(\tau_k)
  +\sum_{j=E(\tau_k)+1}^{E(\tau_{k+1})}w_j1_{\{w_j\ge F(\tau_k)\}}. 
  \label{startTimes1}
\end{equation*}
Using $t=\tau_{k+1}$ in \eqref{WaboveEqn} of Lemma \ref{interStartTimes-n},
rewrite the right side of the display above to obtain
\begin{equation*}
  \langle \chi 1_{[F_p(\sigma_k),\infty)},\mathcal{Z}_p(\tau_{k+1})\rangle
  = \langle\chi 1_{[F(\tau_k),\infty)},\mathcal{Z}(\tau_{k+1}) \rangle.
\end{equation*}
Since $F(\tau_k)\le F(\tau_{k+1})$, the right side is bounded below by $
\langle\chi 1_{[F(\tau_{k+1}),\infty)},\mathcal{Z}(\tau_{k+1})\rangle$, and
bounded above by $W(\tau_{k+1})$, both of which are equal by \eqref{WisAbove}
of Lemma \ref{interStartTimes-n}. Thus,
\begin{equation*}
  \langle \chi 1_{[F_p(\sigma_k),\infty)},\mathcal{Z}_p(\tau_{k+1})\rangle
  = W(\tau_{k+1}),
\end{equation*}
which implies that 
\begin{equation}\label{CisF}
  \langle 1_{[0,F_p(\sigma_k))}, \mathcal{Z}_p(\tau_{k+1})\rangle
 	= \langle\chi 1_{[0,F_p(\sigma_k))}, \mathcal{Z}_p(\tau_{k+1})\rangle
	=0.
\end{equation}
Since $\mathcal{Z}_p(\tau_{k+1})\neq {\bf 0}$, \eqref{CisF} implies that
$C_p(\tau_{k+1})\ge F_p(\sigma_k)$. But $F_p(\cdot)$ is constant on
$[\sigma_k,\sigma_{k+1})$, so $C_p(\tau_{k+1})\ge F_p(\tau_{k+1})$, which
implies by Lemma \ref{CF} that $\tau_{k+1}\in\mathcal{T}_p$, a contradiction.

Now suppose that $\sigma_{k+1}<\tau_{k+1}$. We proceed analogously to the
above argument, with the roles of $\tau_{k+1},\sigma_{k+1}$ and Lemmas
\ref{interStartTimes-n} and \ref{interStartTimes-p} reversed. We have
$\sigma_{k+1}\in(\tau_k,\tau_{k+1})$, so by Lemma \ref{interStartTimes-n}
and the induction hypothesis,
\begin{equation*}
  \langle \chi 1_{[F(\tau_k),\infty)},\mathcal{Z}(\sigma_{k+1})\rangle
  = W(\sigma_k)-F_p(\sigma_k)
  +\sum_{j=E(\sigma_k)+1}^{E(\sigma_{k+1})}w_j1_{\{w_j\ge F_p(\sigma_k)\}}. 
  \label{startTimes1}
\end{equation*}
Using $t=\sigma_{k+1}$ in \eqref{pWaboveEqn} of Lemma \ref{interStartTimes-p},
rewrite the right side of the display above to obtain
\begin{equation*}
  \langle \chi 1_{[F(\tau_k),\infty)},\mathcal{Z}(\sigma_{k+1})\rangle
  = \langle\chi 1_{[F_p(\sigma_k),\infty)},\mathcal{Z}_p(\sigma_{k+1}) \rangle.
\end{equation*}
Since $F_p(\sigma_k)\le F_p(\sigma_{k+1})$, the right side is bounded below by $
\langle\chi 1_{[F_p(\sigma_{k+1}),\infty)},\mathcal{Z}_p(\sigma_{k+1})\rangle$, and
bounded above by $W(\sigma_{k+1})$, both of which are equal by \eqref{pWisAbove}
of Lemma \ref{interStartTimes-p}. Thus,
\begin{equation*}
  \langle \chi 1_{[F(\tau_k),\infty)},\mathcal{Z}(\sigma_{k+1})\rangle
  = W(\sigma_{k+1}),
\end{equation*}
which, since $F(\cdot)$ is constant on $[\tau_k,\tau_{k+1})$, implies that 
\begin{equation}\label{CisF2}
  \langle\chi 1_{[0,F(\sigma_{k+1}))}, \mathcal{Z}(\sigma_{k+1})\rangle
	=0.
\end{equation}
Since $\mathcal{Z}(\sigma_{k+1})\neq {\bf 0}$, \eqref{CisF2} implies that
the set $I=\{j\in J(\sigma_{k+1}): w_j(\sigma_{k+1})\ge
F(\sigma_{k+1})\}$ is nonempty. Let
$j^*=\operatorname{argmin}\{\gamma_{n(j)}:j\in I\}$ be the job in $I$ with the
earliest start time. Then by \eqref{CisF2},
$w_{j_{n(j^*)-1}}(\sigma_{k+1})=0$, which implies that
$\gamma_{n(j^*)-1}+y_{n(j^*)-1}\le\sigma_{k+1}$ and $Z_0+E(\sigma_{k+1})\ge
n(j^*)$. Conclude from \eqref{gamma.y} that $\gamma_{n(j^*)}=\sigma_{k+1}$.
Moreover, the definition of $F(\cdot)$ implies that
$w_{j^*}=F(\sigma_{k+1})$, and so $\sigma_{k+1}\in\mathcal{T}$, a
contradiction.

The previous two arguments imply that $\tau_{k+1}=\sigma_{k+1}$.  We now show
that $F(\tau_{k+1})=F_p(\tau_{k+1})$. Suppose that
$F_p(\tau_{k+1})>F(\tau_{k+1})$ so that the interval
$[F(\tau_{k+1}),F_p(\tau_{k+1}))$ is nonempty. By definition of $\mathcal{T}$,
there is a job $j_*\in\mathcal{J}$ such that $\gamma_{n(j_*)}=\tau_{k+1}$ and
$w_{j_*}=F(\tau_{k+1})$. Then $w_{j_*}(\tau_{k+1})=w_{j_*}$ and so 
\begin{equation}
  \langle\chi 1_{[F(\tau_{k+1}),F_p(\tau_{k+1}))}, 
  				\mathcal{Z}(\tau_{k+1})\rangle >0.
  \label{FisFpcontra}
\end{equation}

Note that for $j\in J(\tau_{k+1})$, $w_j\ge F_p(\tau_{k+1})$ implies that
$w_j>F(\tau_{k+1})$ and so $w_j(\tau_{k+1})=w_j\ge F_p(\tau_{k+1})$.
Conversely, $w_j(\tau_{k+1})\ge F_p(\tau_{k+1})$ clearly implies $w_j\ge
w_j(\tau_{k+1})\ge F_p(\tau_{k+1})$. So 
\begin{equation}\label{FisFp-1}
  w_j(\tau_{k+1})1_{\{w_j(\tau_{k+1})\ge F_p(\tau_{k+1})\}}
=w_j1_{\{w_j\ge F_p(\tau_{k+1})\}}.
\end{equation}
Similarly, $w_j\ge F_p(\tau_{k+1})$ implies $w_j>F(\tau_{k+1})\ge
F(\tau_k)=F_p(\tau_k)$, and so $v_j(\tau_{k+1})=w_j\ge F_p(\tau_{k+1})$.
Conversely, $v_j(\tau_{k+1})\ge F_p(\tau_{k+1})$ gives $w_j\ge
v_j(\tau_{k+1})\ge F_p(\tau_{k+1})$ and so
\begin{equation}\label{FisFp-2}
  v_j(\tau_{k+1})1_{\{v_j(\tau_{k+1})\ge F_p(\tau_{k+1})\}}
=w_j1_{\{w_j\ge F_p(\tau_{k+1})\}}.
\end{equation}
Together, \eqref{FisFp-1} and \eqref{FisFp-2} imply that 
\[
\sum_{j\in J(\tau_{k+1})} 
	w_j(\tau_{k+1})1_{\{w_j(\tau_{k+1})\ge F_p(\tau_{k+1})\}}
	=\sum_{j\in J(\tau_{k+1})} 
	v_j(\tau_{k+1})1_{\{v_j(\tau_{k+1})\ge F_p(\tau_{k+1})\}},
\]
and so 
\begin{equation}
    		\langle\chi 1_{[F_p(\tau_{k+1}),\infty)},
		\mathcal{Z}(\tau_{k+1})\rangle 
	=  \langle\chi 1_{[F_p(\tau_{k+1}),\infty)}, 
    		\mathcal{Z}_p(\tau_{k+1})\rangle.
  \label{FisFpcompare}
\end{equation}
By \eqref{FisFpcompare},
\begin{equation}
  \begin{aligned}
    \langle\chi 1_{[F(\tau_{k+1}),\infty)}, \mathcal{Z}(\tau_{k+1})\rangle
    & = \langle\chi 1_{[F(\tau_{k+1}),F_p(\tau_{k+1}))}, 
    					\mathcal{Z}(\tau_{k+1})\rangle \\
    &\qquad\qquad + 
    		\langle\chi 1_{[F_p(\tau_{k+1}),\infty)},
		\mathcal{Z}(\tau_{k+1})\rangle \\
    & = \langle\chi 1_{[F(\tau_{k+1}),F_p(\tau_{k+1}))}, 
    					\mathcal{Z}(\tau_{k+1})\rangle \\
    &\qquad\qquad + \langle\chi 1_{[F_p(\tau_{k+1}),\infty)}, 
    		\mathcal{Z}_p(\tau_{k+1})\rangle.
  \end{aligned}
  \label{FisFpcrux}
\end{equation}
But the left side and second term on the right side of \eqref{FisFpcrux} are
both equal to $W(\tau_{k+1})$ by Lemmas \ref{interStartTimes-n} and
\ref{interStartTimes-p}. So the first term on the right side equals zero,
contradicting \eqref{FisFpcontra}. This argument is symmetric if the roles of
$F(\tau_{k+1})$ and $F_p(\tau_{k+1})$ are reversed, and we conclude that
$F(\tau_{k+1})=F_p(\tau_{k+1})$.

This completes the proof if $|\mathcal{J}|=|\mathcal{J}_p|$. But this is the
only possibility, for if $|\mathcal{J}|>|\mathcal{J}_p|$, then for
$k=|\mathcal{J}_p|$, there exists $\tau_{k+1}\in\mathcal{J}$ such that
$\tau_{k+1}\in(\sigma_k,\infty)$. This yields a contradiction by the argument
in the paragraph leading to \eqref{CisF} (substituting $\infty$ for
$\sigma_{k+1}$ there). A similar contradiction is implied by
$|\mathcal{J}|<|\mathcal{J}_p|$.
\end{proof}

\begin{corollary}\label{equal frontiers}
Almost surely, $F(\cdot)\equiv F_p(\cdot)$ and for all $t\ge 0$, 
\[ \langle\chi 1_{[F(t),\infty)},\mathcal{Z}(t)\rangle
=\langle\chi 1_{[F(t),\infty)},\mathcal{Z}_p(t)\rangle.\]
\end{corollary}

\begin{proof} The first statement follows from Lemma \ref{startTimes} since
both $F(\cdot)$ and $F_p(\cdot)$ are right-continuous and constant on the
intervals $[\tau_k,\tau_{k+1})$ for all $k<|\mathcal{J}|$ and
$[\tau_k,\infty)$ for $k=|\mathcal{J}|$. The second statement follows from
this fact as well, combined with the first statement, and Lemmas
\ref{interStartTimes-n}, \ref{interStartTimes-p}, and \ref{startTimes}.
\end{proof}

Since the frontier processes are identical, the notation $F(t)$ will be used
hereafter for the frontier process under either policy. We will use the fact
that the state descriptors under either policy are identical above the
frontier.

\begin{corollary}\label{same above}
Almost surely, for all measurable $f:\mathbb{R}_+\to\mathbb{R}_+$ and all $t\ge 0$, 
\[ \langle f 1_{[F(t),\infty)},\mathcal{Z}(t)\rangle
=\langle f 1_{[F(t),\infty)},\mathcal{Z}_p(t)\rangle.\]
\end{corollary}

\begin{proof} 
Fix $t\ge0$. Clearly for all $j\in J(t)$, $w_j(t)>F(t)$ if and only if
$w_j>F(t)$ and $w_j(t)=w_j$, and $v_j(t)>F(t)$ if and only if $w_j>F(t)$ and
$v_j(t)=w_j$. So for all measurable $f$,
\begin{equation}\label{cor2-1}
 \langle f 1_{(F(t),\infty)},\mathcal{Z}(t)\rangle
 =\sum_{j\in J(t)}f(w_j)1_{\{w_j>F(t)\}}
 =\langle f 1_{(F(t),\infty)},\mathcal{Z}_p(t)\rangle.
\end{equation}
Using the special case $f=\chi$ and combining with Corollary 
\ref{equal frontiers} yields 
\begin{equation}
  \langle 1_{\{F(t)\}},\mathcal{Z}(t)\rangle
  =\langle 1_{\{F(t)\}},\mathcal{Z}_p(t)\rangle.
  \label{cor2-2}
\end{equation}
Combining \eqref{cor2-1} and \eqref{cor2-2} completes the proof.
\end{proof}

We will also need the following simple bound for the workload below the frontier in
the SRPT queue. By the workload equation in \eqref{Def1}, 
\begin{align*} W(s) &= W(t) - (V(t)-V(s))+t-s-(I(t)-I(s))\\
	&\le W(t) - (V(t)-V(s))+t-s,
   \end{align*}
almost surely for all $t\ge s\ge0$, since the cumulative idle time process $I(\cdot)$
is nondecreasing. By \eqref{workload-chi}, and by splitting the workload at
times $s$ and $t$ into residual service times that are strictly below the
frontier, and those that are at or above the frontier, 
\begin{multline*} \langle\chi 1_{[0,F(s))},\mathcal{Z}_p(s)\rangle + \langle\chi
1_{[F(s),\infty)},\mathcal{Z}_p(s)\rangle  
    \\ \le  \langle\chi 1_{[0,F(t))},\mathcal{Z}_p(t)\rangle + \langle\chi
1_{[F(t),\infty)},\mathcal{Z}_p(t)\rangle 
- \sum_{j=E(s)+1}^{E(t)}w_j+t-s.
\end{multline*}
Note that $v_j(u)\ge F(u)$ implies $v_j(u)=w_j$ for all $u\ge0$ and $j\in
J(u)$. Thus,
\begin{multline*} \langle\chi 1_{[0,F(s))},\mathcal{Z}_p(s)\rangle +
  \sum_{j\in J(s)}w_j1_{\{v_j(s)\ge F(s)\}} 
  \\ \le  \langle\chi 1_{[0,F(t))},\mathcal{Z}_p(t)\rangle + \sum_{j\in
  J(t)}w_j1_{\{v_j(t)\ge F(t)\}} 
- \sum_{j=E(s)+1}^{E(t)}w_j+t-s.
\end{multline*}
Since $F(\cdot)$ is non-decreasing and residual service times $v_j(\cdot)$ are
non-increasing, $w_j1_{\{v_j(t)\ge F(t)\}}\le w_j1_{\{v_j(s)\ge F(s)\}}$ for
all $j\in J(s)$. Also, $w_j1_{\{v_j(t)\ge F(t)\}}\le w_j$ for all
$j=E(s)+1,\dots,E(t)$. So 
\[ \sum_{j\in J(t)}w_j1_{\{v_j(t)\ge F(t)\}} 
	- \sum_{j=E(s)+1}^{E(t)}w_j - \sum_{j\in J(s)}w_j1_{\{v_j(s)\ge F(s)\}}
    \le 0,
\]
and thus almost surely for all $t\ge s\ge0$, 
\begin{equation} 
  \langle\chi 1_{[0,F(s))},\mathcal{Z}_p(s)\rangle 
   \le  \langle\chi 1_{[0,F(t))},\mathcal{Z}_p(t)\rangle +t-s. \label{dyn ineq}
\end{equation}

\section{Invariance of fluid limit}
\label{sec:limits}

Now we shift our point of view from a single model to a sequence of models.
After defining the sequence and making appropriate asymptotic assumptions, we
will use the work of the previous section to show that the models have the
same fluid limit under SJF and SRPT.

Let $\mathcal{R}$ be a sequence of positive real numbers increasing to
infinity.  For each $r \in \mathcal{R}$, there is an associated stochastic
model with initial condition $Z^r_0$ and $\{w^r_j:j=\dots,-2,-1\}$,
arrival process $E^r(\cdot)$ and service times $\{w_j^r:j=1,2,\dots\}$.  These
primitives have parameters $\alpha^r, \nu^r,$ and $\rho^r$, and are defined on
a probability space $(\Omega^r, \mathcal{F}^r, {\bf P}^r)$.  Then for each $r
\in \mathcal{R}$, the stochastic primitives give rise to measure-valued state
descriptors $\mathcal{Z}^r(\cdot)$ and $\mathcal{Z}_p^r(\cdot)$ under the SJF
and SRPT policies respectively, as well as a common frontier process
$F^r(\cdot)$.  Define fluid-scaled versions of the following processes: 
\begin{equation} \begin{aligned} 
     \bar{F}^r(t)             &= F^r(rt), \\
     \bar{\mathcal{V} }^r(t)      &= \frac{1}{r} \mathcal{V}^r(rt),\\
     \bar{V}^r(t)             &= \frac{1}{r} V^r(rt),\\
     \bar{\mathcal{Z}}^r(t)   &= \frac{1}{r} \mathcal{Z}^r(rt),\\
     \bar{\mathcal{Z}}^r_p(t) &= \frac{1}{r} \mathcal{Z}^r_p(rt), \\
     \bar{W}^r(t)             &= \frac{1}{r} W^r(rt). \\
\end{aligned} \label{Def2}
\end{equation}   

Let $\alpha>0$ and $\nu\in\mathbf{M}$ be a probability measure that does not
charge the origin, such that $\rho=\alpha\langle\chi,\nu\rangle\le 1$.  We
impose the following asymptotic assumptions, as $r \to \infty$, on the
sequence of models.  For the sequence of exogenous arrival processes, assume
that
\begin{equation}\label{Rem,1} 
  \frac{1}{r} E^r(rt) {\Rightarrow} \alpha(\cdot),  
\end{equation}
where $\alpha(t)=\alpha t$ for all $t\ge0$.  For the sequence of service time
distributions, assume that
\begin{equation} 
\nu^r \overset{w}{\longrightarrow} \nu \quad\text{and}\quad
	\{\nu^r: r \in \mathcal{R}\}\quad \text{is uniformly integrable.} 
\label{Rem,2} \end{equation}
It follows from $\eqref{Rem,2}$ that $\rho^r \to \rho$. 

Additionally, assume that as $r \to \infty$,
\begin{equation} 
  \left(\bar{\mathcal{Z}}^r(0), \left\langle \chi, \bar{\mathcal{Z}}^r(0)
  \right \rangle \right)  {\Rightarrow}\left(\mathcal{Z}_0, 
  \left\langle \chi, \mathcal{Z}_0 \right \rangle \right), \label{Rem,3}
\end{equation}
where $\mathcal{Z}_0 \in {\bf M}$ is a random measure satisfying 
  \begin{align} 
   \left\langle \chi, \mathcal{Z}_0\right \rangle &< \infty \quad\text{a.s.,}
   \label{Rem,4}\\ 
   \langle 1_{\{0\}},\mathcal{Z}_0\rangle &=0 \quad\text{a.s.} \label{Rem,5} 
\end{align}

It is shown in \cite{DoGrPu} that, under the assumptions
\eqref{Rem,1}--\eqref{Rem,5} as $r\to \infty$, there is joint convergence
\begin{equation}\label{jointConvergence}
  \left(\bar{\mathcal{V}}^r(\cdot),
  \bar{V}^r(\cdot), \bar{W}^r(0), \bar{\mathcal{Z}}^r_p(\cdot)\right)
  \;{\Rightarrow}\; \left(\mathcal{V}^*(\cdot), V^*(\cdot), W^*(0),
  \mathcal{Z}^*(\cdot)\right).
\end{equation} 
Here, $\mathcal{V}^*(t)=\alpha t\nu$ for all $t\ge0$, and
$\mathcal{Z}^*(\cdot)$ is almost surely an {\em SRPT fluid model solution} for
data $(\alpha,\nu)$ and initial condition $\mathcal{Z}^*(0)$, equal in
distribution to $\mathcal{Z}_0$; see Theorem 5.16 in \cite{DoGrPu} as well as
Section 2.2 in \cite{DoGrPu} for a definition of the fluid model solutions.
Such fluid model solutions are analyzed in detail in \cite{DoGrPu,DoGrPu2}.

Within the proof of \eqref{jointConvergence}, it is assumed by invoking the
Skorohod representation theorem that all random elements are defined on a
common probability space $(\Omega^*, \mathcal{F}^*, {\bf P}^*)$ such that,
almost surely as $r \to \infty$, 
\begin{equation} \big(\tilde{\mathcal{V}}^r(\cdot), \tilde{V}^r(\cdot),
  \tilde{W}^r(0), \tilde{\mathcal{Z}}^r_p(\cdot)\big) {\rightarrow}
  \big(\mathcal{V}^*(\cdot), V^*(\cdot), W^*(0), \mathcal{Z}^*(\cdot)\big),
  \label{Skor1} \end{equation}
uniformly on compact time intervals, where all objects with a tilde denote 
Skorohod representations; see (84) in \cite{DoGrPu}.  As part of this proof,
it is shown that almost surely, 
\begin{equation} \lim_{r \to \infty} \left\langle \chi 1_{[0, \tilde{F}^r(t))},
  \tilde{\mathcal{Z}}_p^r(t) \right\rangle = 0,\quad\text{for all $t\ge0$},
 \label{eqn96} 
\end{equation}
where $\tilde{F}^r(t) = \sup_{s\in[0,t]} \sup
\{x\in\mathbb{R}_+:  \langle 1_{[0, x)},
\tilde{\mathcal{Z}}_p^r(s)\rangle = 0\}$ with
$\sup\varnothing=0$. (See (96) in the proof of Lemma 5.23(i) in \cite{DoGrPu}.
Note that since $\rho \leq 1$, 5.23(i) is the relevant statement.  Also, see
the beginning of Section 5.3 in \cite{DoGrPu}; a fixed $\omega$ is chosen from
an event of probability one in $\Omega^*$.  Since Lemma 5.23 is proved for this fixed
$\omega$, $\eqref{eqn96}$ holds almost surely.)  We will use this fact to
establish a similar statement for the original process
$\bar{\mathcal{Z}}^r_p(\cdot)$, but we need to upgrade the pointwise convergence
to uniform convergence on compact time intervals $[0, T].$  This is shown in
the proof of the following lemma.

\begin{lemma}\label{unif conv} For all $T \in [0, \infty),$   
\begin{equation*} \sup_{t \in [0, T]} 
  	\left \langle \chi 1_{[0, \bar{F}^r(t))},
	\bar{\mathcal{Z}}_p^r(t) \right \rangle\Par 0,\quad\text{as
	$r\rightarrow\infty$}.
\end{equation*}
  \end{lemma} 
  
\begin{proof} Since the Skorohod representations in \eqref{Skor1} are equal
in distribution to the original processes, it suffices to show that,
$\mathbf{P}^*$-almost surely,
\begin{equation}
\lim_{r \to \infty} \sup_{t \in
    [0, T]} \left \langle \chi 1_{[0, \tilde{F}^r(t))},
    \tilde{\mathcal{Z}}_p^r(t) \right \rangle  = 0.
  \label{suffL6}
\end{equation}
To that end, let $\mathcal{D}\subset\Omega^*$ be the event of probability one on which
\eqref{eqn96} holds, and for each $r\in\mathcal{R}$ define the event
\begin{multline*} \mathcal{E}^r = \left\{ \left \langle \chi
  1_{[0, \tilde{F}^r(s))}, \tilde{\mathcal{Z}}_p^r(s) \right \rangle 
  \le \left \langle \chi 1_{[0, \tilde{F}^r(t))},
  \tilde{\mathcal{Z}}_p^r(t) \right \rangle + t - s\right. \\
  \phantom{\Bigl\}} \text{for all $0 \leq s \leq t < \infty$}  \Bigr\}. 
\end{multline*}
Since the dynamic inequality $\eqref{dyn ineq}$ is true almost surely, it also
holds under fluid scaling $\mathbf{P}^r$-almost surely for each
$r\in\mathcal{R}$, and therefore $\mathbf{P}^*$-almost surely for each
$\tilde{\mathcal{Z}}^r_p(\cdot)$, $r\in\mathcal{R}$. Thus,
$\mathbf{P}^*\left( \mathcal{D} \cap \bigcap_{r\in\mathcal{R}}
\mathcal{E}^r\right)=1. $

Fix $T \in [0, \infty)$ and $\epsilon > 0$.  Let $0 = a_0, a_1, ... , a_k = T$
be a partition of $[0, T]$ such that $\max_{j\le k} a_j - a_{j - 1} \le
\epsilon$.  For $t \in [0, T]$, let $a_t = \min\{a_j:a_j\ge t\}$.
Then on $\mathcal{D}$,
\[  \lim_{r\rightarrow\infty}
   \max_{j\le k}\left \langle \chi 1_{[0, \tilde{F}^r(a_j))}, 
   				\tilde{\mathcal{Z}}_p^r(a_j) \right \rangle 
				=0, 
\]
and on $\bigcap_{r\in\mathcal{R}}\mathcal{E}^r$, for all
$r\in\mathcal{R}$ and $t\in[0,T]$,
\begin{align*}
\left \langle \chi 1_{[0, \tilde{F}^r(t))}, \tilde{\mathcal{Z}}_p^r(t)
\right \rangle  &\le  \left \langle \chi 1_{[0,
\tilde{F}^r(a_t))}, \tilde{\mathcal{Z}}_p^r(a_t) \right \rangle +
a_t - t \\
& \le \max_{j\le k}\left\langle\chi 1_{[0,\tilde{F}^r(a_j))},\tilde{\mathcal{Z}}_p^r(a_j)\right\rangle+\epsilon.
\end{align*}
So on  $\mathcal{D}\cap\bigcap_{r\in\mathcal{R}}\mathcal{E}^r$,
\[
\limsup_{r\rightarrow\infty}\sup_{t\in[0,T]}
   \left \langle \chi 1_{[0, \tilde{F}^r(t))}, \tilde{\mathcal{Z}}_p^r(t)
\right \rangle  \le \epsilon, 
\]
which proves \eqref{suffL6}.
\end{proof}

Next, we establish a bound for the fluid scaled mass near the origin, which is
valid under either policy.

\begin{lemma}\label{lowMass}
  Let $T<\infty$. For each $\epsilon,\eta\in (0,1)$ there exists $\delta>0$
  such that 
  \[ \liminf_{r\rightarrow\infty}\mathbf{P}^r\left( \sup_{t\in[0,T]}
  \left(\langle 1_{[0,\delta]},\bar{\mathcal{Z}}^r(t)\rangle
  \vee\langle 1_{[0,\delta]},\bar{\mathcal{Z}}^r_p(t)\rangle\right)
  \le \epsilon \right)\ge 1-\eta.
 \]
\end{lemma}

\begin{proof} 
Fix $\epsilon, \eta\in(0,1)$.  By Lemma 5.11 in \cite{DoGrPu}, there exists
$\delta_1>0$ such that
  \begin{equation*}
	\liminf_{r\rightarrow\infty}\mathbf{P}^r\left(
    	 \sup_{t\in[0,T]} 
    	\langle 1_{[0,\delta_1]},\bar{\mathcal{Z}}^r_p(t)\rangle \le \epsilon
	\right)\ge 1-\frac{\eta}{2}.
  \end{equation*}
By \eqref{Rem,5}, there exists $\delta_2>0$ such that 
\begin{equation}
  \mathbf{P}\left( \langle 1_{[0,\delta_2]},\mathcal{Z}_0\rangle 
  	<\frac{\epsilon}{3} \right)
  \ge 1-\frac{\eta}{2}.
  \label{lowMass2}
\end{equation}
Since $\nu$ puts no mass at the origin, there exists $\delta_3>0$
such that
\begin{equation}\label{lowMass3}
  \alpha T\langle 1_{[0,\delta_3]},\nu\rangle <\frac{\epsilon}{3}.
\end{equation}
Let $\delta=\min\{\delta_1,\delta_2,\delta_3\}$ and define the events
\begin{align*}
	\Omega^r_1 & = \left\{
		\sup_{t\in[0,T]} \langle 1_{[0,\delta]},
			\bar{\mathcal{Z}}^r_p(t)\rangle \le \epsilon\right\}, \\
			\Omega^r_2 & = \left\{  \langle
			1_{[0,\delta]},\bar{\mathcal{Z}}^r(0)\rangle 
  	<\frac{\epsilon}{3} \right\}, \\
	\Omega^r_3 & = \left\{ \langle
	1_{[0,\delta]},\bar{\mathcal{V}}^r(T)\rangle<\frac{\epsilon}{3} \right\}.
\end{align*}
Then $\liminf_{r\rightarrow\infty}\mathbf{P}^r(\Omega^r_1)\ge1-\eta/2$ since
$\delta\le\delta_1$.  The set $\{\xi\in\mathbf{M}:\langle
1_{[0,\delta_2]},\xi\rangle<\epsilon/3\}$ is open in the weak topology. So by
\eqref{Rem,3}, \eqref{lowMass2}, the Portmanteau theorem, and since
$\delta\le\delta_2$, $\liminf_{r\rightarrow\infty}\mathbf{P}^r\left(
\Omega^r_2\right)\ge 1-\eta/2$.
Similarly, by \eqref{jointConvergence}, \eqref{lowMass3}, the Portmanteau
theorem, and since $\delta\le \delta_3$, $\liminf_{r\rightarrow\infty}\mathbf{P}^r\left(
\Omega^r_3\right) =1$. Hence,
\begin{equation*}
  \liminf_{r\to\infty}\mathbf{P}^r\left(
  \Omega^r_1\cap\Omega^r_2\cap\Omega^r_3 \right)\ge 1-\eta.
\end{equation*}
Note that for all $t\ge0$, at most one $j\in J(t)$ satisfies
$0<w_j(t)<w_j$ by \eqref{gamma.y}. So almost surely for all $t\ge0$,
$\sum_{j\in J(t)}1_{(0,\delta]}(w_j(t))1_{(\delta,\infty)}(w_j)\le 1$. Under
fluid scaling, this implies
\[
\sup_{t\in [0,T]} \langle 1_{[0,\delta]},\bar{\mathcal{Z}}^r(t)\rangle
\le \langle 1_{[0,\delta]},\bar{\mathcal{Z}}^r(0)\rangle
+\langle 1_{[0,\delta]},\bar{\mathcal{V}}^r(T)\rangle +\frac{1}{r}.
\]
Then almost surely on $\Omega^r_1\cap\Omega^r_2\cap\Omega^r_3$, 
\begin{equation*}
\sup_{t\in [0,T]} \langle 1_{[0,\delta]},\bar{\mathcal{Z}}^r(t)\rangle
\le \frac{\epsilon}{3}+\frac{\epsilon}{3}+\frac{1}{r},
\end{equation*}
which is bounded above by $\epsilon$ for sufficiently large $r$.
\end{proof}

We are now ready to prove that the fluid scaled state descriptors under SJF
and SRPT converge together. 

\begin{theorem}\label{main theorem} 
For all $T<\infty$, 
\[
\sup_{t\in[0,T]}\dis[\bar{\mathcal{Z}}^r(t),\bar{\mathcal{Z}}_p^r(t)]	
\overset{\mathbf{P}}{\longrightarrow}0, \quad \text{as $r\rightarrow\infty$.}
\]
\end{theorem}

\begin{proof} Fix $T<\infty$ and $\epsilon,\eta>0$.  By Lemma \ref{lowMass},
  there exists $\delta>0$ such that the events
\begin{equation*}
  \Omega^r_1=\left\{ \sup_{t\in[0,T]}\left(\langle
  1_{[0,\delta]},\bar{\mathcal{Z}}^r(t)\rangle
  \vee \langle
  1_{[0,\delta]},\bar{\mathcal{Z}}^r_p(t)\rangle\right)
  \le\frac{\epsilon}{2}\right\}
\end{equation*}
satisfy $\liminf_{r\to\infty}\mathbf{P}^r(\Omega^r_1)\ge 1-\eta/2$. By Lemma 
\ref{unif conv}, the events
\begin{equation*} 
\Omega^r_2 = \left\{  \sup_{t \in [0,T]} 
	\langle \chi 1_{[0, \bar{F}^r(t))}, \bar{\mathcal{Z}}^r_p(t)\rangle 
	\le \frac {\delta \epsilon}{2} \right\} 
\end{equation*}
satisfy $\liminf_{r\to\infty}\mathbf{P}^r(\Omega^r_2)\ge 1-\eta/2$, and so
$\liminf_{r\to\infty}\mathbf{P}^r(\Omega^r_1\cap\Omega^r_2)\ge 1-\eta$.

Fix $t\in[0,T]$ and let $B \subset \R_+$ be closed.  By intersecting $B$ with
$[0,\delta]$, $(\delta, \bar{F}^r(t))$, and $[\bar{F}^r(t),\infty)$, and by
Markov's inequality,
\begin{equation*}
  \left \langle  1_B, \bar{\mathcal{Z}}^r(t) \right \rangle 
  	 \le \left \langle 1_{[0, \delta]}, \bar{\mathcal{Z}}^r(t) \right \rangle
	+ \frac{1}{\delta} \langle \chi 1_{[0, \bar{F}^r(t))}, 
		\bar{\mathcal{Z}}^r(t)  \rangle
  	+ \left \langle 1_{[\bar{F}^r(t)),\infty)\cap B}, \bar{\mathcal{Z}}^r(t)
	\right \rangle.
\end{equation*}
We can replace $\bar{\mathcal{Z}}^r(t)$ by $\bar{\mathcal{Z}}^r_p(t)$ almost
surely in the last term by Corollary \ref{same above}, and in the next to last
term by Corollary \ref{equal frontiers}, since
\[ \langle \chi 1_{[0, \bar{F}^r(t))}, \bar{\mathcal{Z}}^r(t) \rangle 
	= \bar{W}^r(t)- \langle \chi 1_{[\bar{F}^r(t),\infty)},
	\bar{\mathcal{Z}}^r(t)\rangle.
\]
Then use $[\bar{F}^r(t)),\infty)\cap B\subset B^\epsilon$ to get, almost surely on
$\Omega^r_1\cap\Omega^r_2$, 
\begin{equation}
  \langle  1_B, \bar{\mathcal{Z}}^r(t) \rangle 
  \le \frac{\epsilon}{2} + \frac{\epsilon}{2}
  	+ \langle 1_{B^\epsilon}, \bar{\mathcal{Z}}^r_p(t)
	 \rangle.
	\label{distance1}
\end{equation}
The same argument yields, almost surely on $\Omega^r_1\cap\Omega^r_2$, 
\begin{equation}
  \langle  1_B, \bar{\mathcal{Z}}^r_p(t)  \rangle 
  \le \epsilon + \langle 1_{B^\epsilon}, \bar{\mathcal{Z}}^r(t) \rangle.
	\label{distance2}
\end{equation}
Since $t\in[0,T]$ and $B\subset\mathbb{R}_+$ were arbitrary, \eqref{distance1}
and \eqref{distance2} imply that, almost surely on $\Omega^r_1\cap\Omega^r_2$, 
\begin{equation*}
  \sup_{t\in[0,T]}\dis[\bar{\mathcal{Z}}^r(t),\bar{\mathcal{Z}}^r_p(t)]\le\epsilon.
\end{equation*}
\end{proof}

\begin{proof}[Proof of Theorem \ref{upshot}.] 
By \eqref{jointConvergence}, $\bar{\mathcal{Z}}^r_p(\cdot) \Rightarrow
\mathcal{Z}^*(\cdot)$ as $r\to\infty$.   Theorem \ref{upshot} thus follows
from Theorem \ref{main theorem} and the converging together lemma (see e.g.\
Theorem 4.1 in \cite{Bi1968}). 
\end{proof}

\begin{center}
{\bf Acknowledgement}
\end{center}
We are grateful to Rami Atar for suggesting this project and conjecturing the
result.


\vspace{2ex}
\begin{minipage}{2in}
\footnotesize {\sc Department of Mathematics\\
University of Virginia\\
Charlottesville, VA 22904\\
E-mail:} gromoll@virginia.edu\\
\phantom{E-mail: } mpk2v@virginia.edu
\end{minipage} 

\end{document}